\documentclass[a4paper,reqno,12pt]{amsart}
\usepackage{t1enc}
\usepackage[margin=1.0in]{geometry}
\usepackage{ifthen}
\usepackage{graphicx}
\usepackage{amsmath,amstext,amssymb,amsopn,amsthm,color}
\usepackage[normalem]{ ulem }
\usepackage{soul}

\newcommand{\expr}[1]{\left( #1 \right)}

\newcommand{\laplace}{\mathcal{L}}

\newcommand{\C}{\mathbf{C}}
\newcommand{\R}{\mathbf{R}}

\newcommand{\pr}{\mathbb{P}}
\newcommand{\ex}{\mathbf{E}}

\newcommand{\ind}{\mathbf{1}}

\newcommand{\ph}{\varphi}

\theoremstyle{plain}
\newtheorem{theorem}{Theorem}

\newtheorem{corollary}{Corollary}
\newtheorem{proposition}{Proposition}

\theoremstyle{definition}

\newtheorem{remark}{Remark}

\theoremstyle{remark}

\DeclareMathOperator{\re}{Re}
\DeclareMathOperator{\im}{Im}
\DeclareMathOperator{\Arg}{Arg}

\newcommand{\formula}[2][nolabel]
{\ifthenelse{\equal{#1}{nolabel}}
 {\begin{align*} #2 \end{align*}}
 {\ifthenelse{\equal{#1}{}}
  {\begin{align} #2 \end{align}}
  {\begin{align} \label{#1} #2 \end{align}}
 }
}

\numberwithin{equation}{section}

\begin{document}

%
%

\title[The entrance law of the excursion measure of the reflected process]{The entrance law of the
excursion measure of the reflected process for some classes of L\'evy processes}

\author{Lo\"i{}c Chaumont and Jacek Ma{\l}ecki}
\address{Lo\"i{}c Chaumont, \\ LAREMA, D\'e{}partement de Math\'e{}matique\\ Universit\'e d'Angers \\ Bd Lavoisier - 49045, Angers Cedex 01, France}
\email{loic.chaumont@univ-angers.fr}
\address{Jacek Ma{\l}ecki\\ Faculty of Pure and Applied Mathematics \\ Wroc{\l}aw University of Science and Technology \\ ul. Wybrze{\.z}e Wyspia{\'n}\-skiego 27 \\ 50-370 Wroc{\l}aw, Poland}
\email{jacek.malecki@pwr.edu.pl}


\thanks{J. Ma{\l}ecki is supported by the Polish National Science Centre (NCN) grant no. 2015/19/B/ST1/01457.}

\begin{abstract} We provide integral formulae for the Laplace transform of the entrance law of the reflected excursions for symmetric L\'evy processes in terms of their characteristic exponent. For subordinate Brownian motions and stable processes we express the density of the entrance law in terms of the generalized eigenfunctions for the semigroup of the process killed when exiting the positive half-line. We use the formulae to study in-depth properties of the density of the entrance law such as asymptotic behavior of its derivatives in time variable.
\end{abstract}

\maketitle

\section{Introduction}
\label{section:introduction}

It follows from excursion theory that the trajectories of a L\'evy process can be decomposed using the excursions of 
the process reflected in its past infimum. This result justifies the importance of knowing the excursion measure of the reflected 
process and more particularly, the entrance law of this measure. There are also several interesting applications of this entrance 
law. First it is directly related to the potential measure of the time space ladder height process, see Lemma 1 in \cite{bib:ch13}. 
Moreover it provides a useful  expression of the distribution density of the  supremum of the L\'evy process before fixed times, 
\cite{bib:ch13}, \cite{bib:chm16}. More recently it has been involved in the study of the probability of creeping through curves of 
L\'evy processes, \cite{bib:chp19}.

In this article we obtain integral representations of the densities and the Laplace transforms of the entrance laws of the 
reflected excursions for two classes of real valued L\'evy processes. The first class consists of symmetric L\'evy processes, with a 
particular emphasis on subordinate Brownian motions, when the L\'evy measure of the underlying subordinator has a completely monotone density. The other class is that of stable processes. The presented formulae for symmetric 
processes are given 
in terms of the corresponding L\'evy-Kchintchin exponent $\Psi(\xi)$ and the related generalized eigenfunctions introduced by 
M. Kwa\'snicki in \cite{bib:mk10}. In the stable case, we based the calculations on the generalized eigenfunctions studied 
recently by 
A. Kuznetsov and M. Kwa\'snicki in \cite{bib:kk18}. Then we used the formulae obtained for the entrance law densities to 
derive corresponding 
integral representations for supremum densities. Although the theory of L\'evy processes is very rich and abounds in numerous general relationships, as those coming from the Wiener-Hopf factorizations, there are few examples where the explicit 
representations of the 
related densities are available. Apart from Brownian motion and Cauchy process, some series representations were recently 
found in \cite{bib:HubalekKuznetsov:2011}, \cite{bib:ak13}, \cite{bib:HackmannKuznetsov:2013} in the case of stable processes. 
A different approach was presented in \cite{bib:kmr13}, where the theory of  Kwa\'snicki's generalized eigenfunctions were used to described the first passage time density through a barrier for subordinate Brownian motions with regular L\'evy measures. This concept was generalized to non-symmetric stable processes in \cite{bib:kk18}. In the present paper we stay in this framework and show that a similar approach leads to integral representations of the entrance law density, the supremum density or the density of 
joint distribution of the process itself and its supremum. Then we apply the obtained formulae to study the asymptotic behavior 
of the derivatives in time variable of the entrance law densities of the reflected excursions. Let us finally mention that these formulae can be used to perform numerical simulations and study in-depth properties of the process coupled with its past supremum. 
 
\section{Preliminaries}
\label{section:preliminaries}

Let $(X,\pr)$ be the real valued L\'evy process whose characteristic exponent $\Psi(\xi)$ is characterized in terms 
of the L\'evy triplet  $(a,\sigma^2,\Pi)$ by the L\'evy-Kchintchin formula
\begin{equation*}
    \Psi(\xi) = -ai\xi+\frac12 \sigma^2 \xi^2 -\int_{\R\setminus\{0\}}\left(e^{i\xi x}-1-i\xi\ind_{\{|x|<1\}}\right)\Pi(dx)\/,\quad \xi\in\R\/.
\end{equation*}
We write $\pr_x$ for the law of the process starting from $x\in\R$. We denote by $X^*=-X$ the dual process and $\pr_x^*$ stands for its law with respect to $\pr_x$. The past supremum and past infimum of $X$ before a deterministic time $t\geq 0$ are
\begin{eqnarray*}
  \overline{X}_t= \sup\{X_s;0\leq s\leq t\}\/, \quad \underline{X}_t= \inf\{X_s;0\leq s\leq t\}\/.
\end{eqnarray*}
For given $t>0$ we write $f_t(dx)=\pr(\overline{X}_t\in dx)$ for the corresponding distribution and $f_t(x)$ stands for its density with respect 
to the Lebesgue measure on $(0,\infty)$ whenever it exists. Recalling that the reflected processes $\overline{X}-X$ and $X-\underline{X}$ 
are Markovian, we write $L_t$ and $L_t^*$ for their local times at $0$ respectively, where these are normalized in the following way
\begin{eqnarray*}
   \ex \left(\int_0^\infty e^{-t}dL_t \right) = \ex \left(\int_0^\infty e^{-t}dL_t^* \right)=1\/.
\end{eqnarray*}
We write $n$ (and $n^*$) for the It\^o measure of the excursions away from $0$ of the reflected process $\overline{X}-X$ (resp. $X-\underline{X}$). Our main objects of studies are the corresponding entrance laws defined by 
\begin{eqnarray*}
    q_t(dx) = n(X_t\in dx,t<\zeta)\/,\quad q_t^*(dx) = n^*(X_t\in dx,t<\zeta)\/,\quad t>0\/,
\end{eqnarray*}
where $\zeta$ is the life time of the  generic excursion and $q_t(x)$, $q_t^*(x)$ denotes the 
densities on $(0,\infty)$ of  $q_t(dx)$ and $q_t^*(dx)$, whenever they exist. In this paper, it will always be 
assumed that 0 is regular for both half-lines $(-\infty,0)$ and $(0,\infty)$. In this case, the double Laplace transform of $q_t(dx)$ is 
given by 
\begin{equation}\label{0000}
\int_0^\infty \int_0^\infty e^{-\xi x}e^{-zs}q_s(dx)ds = \frac{1}{\kappa(z,\xi)}\/,
\end{equation}
where $\kappa(z,\xi)$ is the Laplace exponent of the ladder process $(L_t^{-1},H_t)$, $t<L(\infty)$. 
Here $L_t^{-1}$ denotes the right continuous inverse of $L_t$ (ladder time process) and the ladder-height process is defined by $H_t = X_{L_t^{-1}}$. Analogous relations hold for $q_t^*(dx)$ and the Laplace exponent $\kappa^*(z,\xi)$ for the ladder process $((L_t^*)^{-1},H_t^*)$. Formula (\ref{0000}) actually shows that $q_s(dx)ds$ is the potential measure of $(L^{-1},H)$. We denote by $h$ the renewal function of the ladder height process $H$, that is 
$$
   h(x) = \int_0^\infty \pr(H_t\leq x)\,dt\/,\quad x\geq 0\/.
$$

In the light of Theorem 6 in \cite{bib:ch13}, the entrance laws $q_t(dx)$ and $q_t^*(dx)$ seem to be basic objects in the study of the supremum distributions. More precisely, under our assumption that $0$ is regular for both negative and positive half-lines, the representation (4.4) from \cite{bib:ch13} reads as
\begin{eqnarray}
\label{eq:ftpt:formula}
   \pr(\overline{X}_t\in dx,\overline{X}_t-X_t\in dy) = \int_0^t q_s^*(dx)q_{t-s}(dy)ds\/,
\end{eqnarray}
which, in particular, implies
\begin{equation}
  \label{eq:ft:formula}
	   f_t(dx) = \int_0^t n(t-s<\zeta)\,q_s^*(dx)\/.
\end{equation}

Finally, for $x>0$, we denote by $\mathbb{Q}^*_x$ the law of the processes killed when exiting the positive half-line, i.e. 
$$
   \mathbb{Q}_x^*(\Lambda,t<\zeta) = \pr_x(\Lambda,t<\tau_0^{-})\/,\quad \Lambda\in\mathcal{F}_t\/,
$$
where $\tau_0^-=\inf\{t>0: X_t<0\}$. The law $\mathbb{Q}_x$ is defined in the same way, but with respect to the dual process. The corresponding semigroups are defined as
$$
  \textbf{Q}_t^*f(x) = \mathbb{Q}_x^*f(X_t)\/,\quad \textbf{Q}_tf(x) = \mathbb{Q}_x f(X_t),
$$
for non-negative Borel functions $f$. We also write $q_t^*(x,dy)$, $q_t(x,dy)$ and $q_t^*(x,y)$, $q_t(x,y)$ for the corresponding transition probability measures and their densities whenever they exist. Recall that whenever $q_t^*(x,\cdot)$ and $q_t(x,\cdot)$ are absolutely continuous, the duality relation holds
$$
   q_t^*(x,y) = q_t(y,x)\/.
$$

\section{Symmetric L\'evy processes and subordinated Brownian motions}
\label{sec:sym}

This section is devoted to symmetric L\'evy processes with some addition regularity assumptions on the L\'evy-Kchintchin exponent $\Psi(\xi)$ presented in details below. We also exclude compound Poisson processes from our considerations. Note that the symmetry assumptions simplify the general exposure presented in Preliminaries, where, roughly speaking, we can remove the notation with $*$. Moreover, the ladder time process is the $1/2$-stable subordinator for every symmetric L\'evy process, which implies that 
\begin{equation}
\label{eq:sym:n}
    n(t<\zeta) = \frac{t^{-1/2}}{\sqrt{\pi}}\/,\quad t>0\/.
\end{equation}
Finally, we recall the integral representation of the Laplace exponent of the ladder process
\begin{equation}
\label{eq:kappa:formula}
   \kappa(z,\xi) = {\sqrt{z}}\exp\left(\frac{1}{\pi}\frac{\xi\log(1+\frac{\Psi(\zeta)}{z})}{\xi^2+\zeta^2}\,d\zeta\right)\/,\quad 
	z,\xi\geq 0\/,
\end{equation}
where, in the symmetric case, $\Psi(\xi)$ is a real-valued function. 

Our first result gives the expression for the Laplace transform of $q_t(dx)$ (for fixed $t>0$) in the case of symmetric L\'evy processes with increasing L\'evy-Khintchin exponent.  This is an analogue of Theorem 4.1 in \cite{bib:kmr11}, 
where the corresponding formula for $\overline{X}_t$ was derived. Note that even though the formulae for the Laplace transforms of 
$q_t(dx)$ and $\pr(\overline{X}_t\in dx)$ seem to be similar, passing from one to the other by using \eqref{eq:ft:formula} 
and \eqref{eq:sym:n} is not straightforward.

 \begin{theorem}
\label{thm:sym:Laplace}
Let $(X,\mathbb{P})$ be a symmetric L\'e{}vy process that is not a compound Poisson process.
Assume that the L\'evy-Khintchin exponent $\Psi(\xi)$ of $(X,\mathbb{P})$ is increasing in $\xi>0$. Then
\begin{equation}
\label{eq:qt:laplaceformula}
\int_0^\infty e^{-\xi x}q_t(dx) = \frac1\pi \int_0^\infty \frac{\lambda\Psi'(\lambda)}{\lambda^2+\xi^2}\exp\left(\frac1\pi \int_0^\infty \frac{\xi \log\frac{\lambda^2-u^2}{\Psi(\lambda)-\Psi(u)}}{\xi^2+u^2}du\right)e^{-t\Psi(\lambda)}d\lambda\/.
\end{equation} 
\end{theorem}
\begin{proof}
  The proof is based on the same idea as in the proof of Theorem 4.1 in \cite{bib:kmr11} with a slight modification of the arguments. For the completeness of the exposure and the convenience of the reader we present it below. We put $\psi(\xi) = \Psi(\sqrt{\xi})$ for $\xi > 0$ and define 
  \formula{
     \ph(\xi, z) & =  \exp\expr{-\frac{1}{\pi} \int_0^\infty \frac{\xi \log (1 + \frac{\psi(\zeta^2)}{z})}{\xi^2 + \zeta^2} \, d\zeta}\/,
  }
	which is $\sqrt{z}(\kappa(z,\xi))^{-1}$ by \eqref{eq:kappa:formula}. Obviously, for fixed $\xi>0$, the function $\ph(\xi,z)$ is a holomorphic function of $z$, which is positive for $z>0$. Note also that $\lim_{z\to 0+}\ph(\xi, z) = 0$ (by monotone convergence) and $\lim_{z\to \infty}\ph(\xi, z) = 1$ (by dominated convergence). Moreover, as it was shown in \cite{bib:kmr11}, that for $\im z>0$ we have 
  \formula{
     \Arg\ph(\xi,z) = -\frac{1}{\pi}\int_0^\infty \frac{\xi \Arg(1+\psi(\zeta^2)/z)}{\xi^2+\zeta^2}\,d\zeta \in (0,\pi/2)\/.
  }
  Thus $\Arg(\sqrt{z}\ph(\xi,z))\in (0,\pi)$ for $\im z>0$. This is equivalent to $h_\xi(z) = \ph(\xi,z)/\sqrt{z}$ being a Stieltjes function (for fixed $\xi$). In general, a function $g(z)$ is said to be a Stielties function if 
\begin{equation}
\label{eq:stiel:rep}
   g(z) = \frac{b_1}{z}+b_2+\frac1\pi \int_0^\infty \frac{1}{z+\zeta}\,{\nu(d\zeta)}\/,\quad z\in \C\setminus(-\infty,0]\/,
 \end{equation}
 where $b_1,b_2\geq 0$ and $\nu(d\zeta)$ is a Radon measure on $(0,\infty)$ such that $\int\min(1,\zeta^{-1})\mu(d\zeta)<\infty$. The constants and a measure appearing in the definitions of Stielties functions are given by
\begin{equation}
 \label{eq:stiel:limits}
   b_1 = \lim_{z\to 0+}zg(z)\/,\quad b_2 = \lim_{z\to \infty}g(z)\/, \quad \nu(d\zeta) = \lim_{\varepsilon\to 0+}\im(-g(-\zeta+i\varepsilon)d\zeta)\/.
 \end{equation}
Note that the last limit is understood in the sense of weak limit of measures. Since
	\begin{eqnarray*}
	   \lim_{z\to 0+}z h_\xi(z) = \lim_{z\to 0+}\sqrt{z}\ph(\xi,z) = 0\/,\quad \lim_{z\to 0+} h_\xi(z) = \lim_{z\to 0+}\ph(\xi,z)/\sqrt{z} = 0
	\end{eqnarray*}
	the constants appearing in the representation (\ref{eq:stiel:rep}) for Stielties function $h_\xi(z)$ are zero. Moreover, for $z=\psi(\lambda^2)$ we get 
  \begin{eqnarray*}
  h^+_\xi(-z) &=& \lim_{\varepsilon\to 0^+} h_\xi(-z+i\varepsilon)\\
	&=& \frac{i}{\sqrt{\psi(\lambda^2)}}\frac{\lambda (\lambda + \xi i) }{\lambda^2 + \xi^2}\exp\left(\frac1\pi \int_0^\infty \frac{\xi \log\frac{\psi(\lambda^2)}{\lambda^2}\frac{\lambda^2-u^2}{\psi(\lambda^2)-\psi(u^2)}}{\xi^2+u^2}du\right)\\
   &=& \frac{ \lambda i - \xi  }{\lambda^2 + \xi^2}\exp\left(\frac1\pi \int_0^\infty \frac{\xi \log\frac{\lambda^2-u^2}{\psi(\lambda^2)-\psi(u^2)}}{\xi^2+u^2}du\right)\/.
  \end{eqnarray*}
   Therefore, by (\ref{eq:stiel:limits}), for every $z>0$ we have
\begin{eqnarray*}
 \frac{\ph(\xi, z)}{\sqrt{z}} & = & \frac{1}{\pi} \int_0^\infty \im h^+_\xi(-\zeta) \, \frac{1}{z + \zeta} \, d\zeta \\
 &=&\frac{1}{\pi} \int_0^\infty 2 \lambda \psi'(\lambda^2) \im h^+_\xi(-\psi(\lambda^2)) \, \frac{1}{z + \psi(\lambda^2)} \, d\lambda \\
 &=& \frac{2}{\pi} \int_0^\infty \lambda \psi'(\lambda^2) \frac{\lambda}{\lambda^2 + \xi^2} \, \frac{1}{z + \psi(\lambda^2)}\exp\left(\frac1\pi \int_0^\infty \frac{\xi \log\frac{\lambda^2-u^2}{\psi(\lambda^2)-\psi(u^2)}}{\xi^2+u^2}du\right) \, d\lambda \\
 &=& \int_0^\infty e^{-tz} \left(\frac{1}{\pi}\int_0^\infty \frac{\lambda\Psi'(\lambda)}{\lambda^2+\xi^2}\exp\left(\frac1\pi \int_0^\infty \frac{\xi \log\frac{\lambda^2-u^2}{\Psi(\lambda)-\Psi(u)}}{\xi^2+u^2}du\right)e^{-t\Psi(\lambda)}d\lambda\right)dt\/.
\end{eqnarray*}
Thus, the Laplace transform of the right-hand side of (\ref{eq:qt:laplaceformula}) is equal to $1/\kappa(z,\xi)$ and the theorem follows from uniqueness of the Laplace transform.
\end{proof}

\medskip

From now on, for the rest of the section, we will follow the approach presented in \cite{bib:kmr13} and restrict our consideration to the case where $(X,\mathbb{P})$ is a subordinate Brownian motion whose underlying subordinator has a complete monotone density. 
The process $(X,\mathbb{P})$ has the latter form if and only if its characteristic exponent $\Psi(\xi)$ can be written as $\Psi(\xi)=\psi(\xi^2)$
for a complete Bernstein function $\psi$ (see Proposition 2.3 in \cite{bib:mk10}). A function $\psi(\xi)$ is called a \textit{complete Bernstein function} (CBF) if
 \begin{eqnarray}
\label{eq:cbf:rep}
   \psi(z) = a_1+a_2z+\frac1\pi \int_0^\infty \frac{z}{z+\zeta}\frac{\mu(d\zeta)}{\zeta}\/,\quad z\in \C\setminus(-\infty,0]\/,
 \end{eqnarray}
 where $a_1\geq 0$, $a_2 \geq 0$ and $\mu(d\zeta)$ is a Radon measure on positive half-line such that $\int\min(\zeta^{-1},\zeta^{-2})\mu(d\zeta)$ is finite. As in the Stielties function representation, the above-given constants and the measure $\mu$ are determined by suitable 
limits as follows
 \begin{equation}
\label{eq:cbf:limits}
   a_1 = \lim_{z\to 0+}\psi(z)\/,\quad a_2 = \lim_{z\to \infty}\frac{\psi(z)}{z}\/, \quad \mu(d\zeta) = \lim_{\varepsilon\to 0+}\im(\psi(-\zeta+i\varepsilon)d\zeta)\/.
 \end{equation}

The spectral theory of subordinate Brownian motion on a half-line was developed by M.~Kwa\'snicki in \cite{bib:mk10}, where the generalized eigenfunctions $F_\lambda(x)$ of the transition semigroup $\textbf{Q}_t$ of the process $(X,\mathbb{P})$ killed upon leaving the half-line $[0,\infty)$ were constructed. Some additional properties of $F_\lambda(x)$ were also studied in \cite{bib:kmr13}. For a fixed CBF $\psi$ 
and $\lambda>0$ the generalized eigenfunctions of $\textbf{Q}_t$ with eigenvalue $e^{-t\psi(\lambda^2)}$ are given by 
\begin{equation}
   \label{eq:Fl:definition}
  F_\lambda(x) = \sin(x\lambda+\vartheta_\lambda)-G_\lambda(x)\/,\quad x>0
\end{equation}
where the phase shift $\vartheta_\lambda$ belongs to $[0,\pi/2)$ and is given by
\begin{eqnarray*}
   \vartheta_\lambda = -\frac{1}{\pi}\int_0^\infty \frac{\lambda}{\lambda^2-u^2}\log\frac{\psi'(\lambda^2)(\lambda^2-u^2)}{\psi(\lambda^2)-\psi(u^2)}du\/,\quad \lambda>0\/.
\end{eqnarray*}
Recall the following upper-bounds (Proposition 4.3  and Proposition 4.5 in \cite{bib:kmr13} respectively)
\begin{equation}
\label{eq:theta:upbound1}
\vartheta_\lambda\leq \left(\sup_{\xi>0}\frac{\xi|\psi''(\xi)|}{\psi'(\xi)}\right)\frac{\pi}{4}\/.
\end{equation}
\begin{equation}
\label{eq:theta:upbound2}
\vartheta_\lambda\leq \frac{\pi}{2}-\arcsin\sqrt{\lambda^2\frac{\psi'(\lambda^2)}{\psi(\lambda^2)}}\/,\quad \lambda>0\/.
\end{equation}
The function $G_\lambda$ is the Laplace transform of the finite measure
\begin{equation*}
  \gamma_\lambda(d\xi) = \frac{1}{\pi} \left(\im\frac{\lambda \psi'(\lambda^2)}{\psi(\lambda^2)-\psi^+(\xi^2)}\right)\exp\left(-\frac{1}{\pi}\int_0^\infty \frac{\xi}{\xi^2+u^2}\log\frac{\psi'(\lambda^2)(\lambda^2-u^2)}{\psi(\lambda^2)-\psi(u^2)}du\right)\,d\xi\/.
\end{equation*}
Here $\psi^+$ denotes the holomorphic extension of $\psi$ in the complex upper half-plane. The Laplace transform of $F_\lambda(x)$ is given by 
\begin{equation}
\label{eq:Fl:laplacetransform}
   \laplace{F}_\lambda(\xi) =\frac{\lambda}{\lambda^2+\xi^2}\exp\left(\frac{1}{\pi}\int_0^\infty\frac{z}{z^2+u^2}\log\frac{\psi'(\lambda^2)(\lambda^2-u^2)}{\psi(\lambda^2)-\psi(u^2)}du\right)\/.
	\end{equation}
Recall also the following estimates 
\begin{equation}
\label{eq:Fl:laplaceestimate}
   |\laplace{F}_\lambda(\xi)|\leq \frac{|\lambda+\xi|}{|\lambda^2+\xi^2|}\/,\quad x>0\/,\re\xi>0\/.
\end{equation}   
Proposition 5.4 in \cite{bib:kmr13} states that for unbounded $\psi$ such that $\limsup_{\lambda\to0^+}{\vartheta_\lambda}<\pi/2$ we have the following limiting behavior 
\begin{equation}
\label{eq:Fl:limit}
\lim_{\lambda\to 0^+}\frac{F_{\lambda}(x)}{\lambda\sqrt{\psi'(\lambda^2)}} = h(x)\/,\quad x \geq 0\/,
\end{equation}
and the convergence is locally uniform in $x\geq 0$.

The functions $F_\lambda(x)$ were used to find the integral representations for the density function of $\tau_0^-$ and its derivatives (see Theorem 1.5 in \cite{bib:kmr13}). In the next Theorem we show that an analogous representation can be obtained for the density of the entrance law. 

\begin{theorem}
\label{thm:qt:formula}
  Let $(X,\mathbb{P})$ be a symmetric L\'e{}vy process whose L\'evy-Khintchin exponent $\Psi(\xi)$ satisfies $\Psi(\xi) = \psi(\xi^2)$ 
  for a complete Bernstein function $\psi(\xi)$. If there exists $t_0>0$ such that 
  \begin{equation}
	\label{eq:qt:integcond}
     \int_1^\infty e^{-t_0\psi(\lambda^2)}\lambda \sqrt{\psi'(\lambda^2)}d\lambda <\infty \/,
  \end{equation}
  then, for every $t\geq t_0, $ $q_t(dx)$ has a density with respect to the Lebesgue measure given by the formula
  \begin{equation}
	\label{eq:qt:formula}
    q_t(x) = \frac{2}{\pi}\int_0^\infty e^{-t\psi(\lambda^2)}F_\lambda(x)\lambda \sqrt{\psi'(\lambda^2)}d\lambda\/,\quad x>0\/.
  \end{equation}
\end{theorem}
\begin{proof}
   The definition \eqref{eq:Fl:definition} of $F_\lambda(x)$ and the fact that $G_\lambda(x)$ is a Laplace transform of finite measure and $G_\lambda(0)=\sin(\vartheta_\lambda)$ entail that $|F_\lambda(x)|\leq 2$ for all $x, \lambda>0$. The assumption (\ref{eq:qt:integcond}) together with the estimate
  \begin{equation*}
      \int_0^1 e^{-t\psi(\lambda^2)}\lambda\sqrt{\psi'(\lambda^2)}d\lambda \leq \frac{1}{\sqrt{\psi'(1)}}\int_0^1 e^{-t\psi(\lambda^2)}\lambda{\psi'(\lambda^2)}d\lambda
      =\frac{1-e^{-t\psi(1)}}{2t\sqrt{\psi'(1)}}
   \end{equation*}
give that the function $e^{-\xi x}e^{-t\psi(\lambda^2)}|F_\lambda(x)|\lambda\sqrt{\psi'(\lambda^2)}$ is jointly integrable on $(\lambda,x)\in(0,\infty)^2$ and consequently the Laplace transform of the integral appearing on the right-hand side of (\ref{eq:qt:formula}) is given by
\begin{eqnarray*}
  \int_0^\infty e^{-t\psi(\lambda^2)}\laplace \lefteqn{F_\lambda(\xi)\lambda \sqrt{\psi'(\lambda^2)}d\lambda }\\
	&=&\int_0^\infty \frac{\lambda^2\psi'(\lambda^2)}{\lambda^2+\xi^2}\exp\left(\frac1\pi \int_0^\infty \frac{\xi \log\frac{\lambda^2-u^2}{\psi(\lambda^2)-\psi(u^2)}}{\xi^2+u^2}du\right)e^{-t\psi(\lambda^2)}d\lambda\/.
\end{eqnarray*}
This is just the right-hand side of \eqref{eq:qt:laplaceformula} with $\Psi(\xi)=\psi(\xi^2)$ divided by $2/\pi$. The uniqueness of the Laplace transform ends the proof.
\end{proof}
It is very easy to see that if $\psi$ is regularly varying at infinity with strictly positive order then the exponential factor in (\ref{eq:qt:integcond}) makes the integral convergent for every $t>0$. Thus we derive the following result.
\begin{corollary}
If $\psi$ is CBF regularly varying at infinity with order $\alpha\in(0,1]$, then the measure $q_t(dx)$ has a density for every $t>0$ and the formula (\ref{eq:qt:formula}) holds for every $t>0$ and $x>0$.
\end{corollary}
\begin{remark}
The condition (\ref{eq:qt:integcond}) is satisfied for a large class of CBFs like $\psi(\xi) = \xi^{\alpha/2}$, $\alpha\in(0,2)$ (symmetric stable), $\psi(\xi)=\xi^{\alpha/2}+\xi^{\beta/2}$, $\alpha,\beta\in(0,2)$ (sum o two independent stable), $\psi(\xi) = (m^2+\xi)^{\alpha/2}-m$, $\alpha\in(0,2)$ (relativistic stable), $\psi(\xi) = \log(1+\xi^{\alpha/2})$, $\alpha\in(0,2]$, $t>1/\alpha$ (geometric stable)
\end{remark}

Then we obtain the following straightforward consequence of the previous result.
\begin{theorem}
\label{thm:sym:distributions} Let $(X,\mathbb{P})$ be a symmetric L\'e{}vy process whose L\'evy-Khintchin exponent $\Psi(\xi)$ satisfies $\Psi(\xi) = \psi(\xi^2)$ for a complete Bernstein function $\psi(\xi)$. If \eqref{eq:qt:integcond} holds, then for every $t>t_0$ the distribution of 
$(\overline{X}_t,\overline{X}_t-X_t)$ is absolutely continuous with respect to the Lebesgue measure on $(0,\infty)^2$ with  density 
\begin{equation*}
   \frac{4}{\pi^2}\iint_{(0,\infty)^2}\frac{e^{-t\psi(\lambda^2)}-e^{-t\psi(u^2)}}{\psi(\lambda^2)-\psi(u^2)}F_\lambda(x)F_u(y)\lambda \,u\sqrt{\psi'(\lambda^2)\psi(u^2))}du\,d\lambda\/.
\end{equation*}
Moreover, we have
\begin{equation*}
   f_t(x) = \frac{2}{\pi^{3/2}}\int_0^\infty e^{-t\psi(\lambda^2)}\left(\int_0^{t\psi(\lambda^2)}\frac{e^udu}{\sqrt{u}}\right)F_\lambda(x)\,\lambda\sqrt{\psi'(\lambda^2)}\,d\lambda\/,
\end{equation*}
for every $t>t_0$.
\end{theorem}
\begin{proof}
   The proofs of both formulae are direct consequences of the integral representation \eqref{eq:qt:formula}, the relations \eqref{eq:ftpt:formula}, \eqref{eq:ft:formula} and \eqref{eq:sym:n} together with the Fubini's theorem, which can be applied due to the integral condition \eqref{eq:qt:integcond}.
\end{proof}

The representation \eqref{eq:qt:formula} enables to compute the derivatives of $q_t(x)$ and examine its behavior in two asymptotic regimes: as $t$ goes to infinity and $x$ goes to $0$. It is described in the following theorem.
\begin{theorem}
\label{thm:qt:derivative}
Let $(X,\mathbb{P})$ be a symmetric L\'e{}vy process whose L\'evy-Khintchin exponent $\Psi(\xi)$ satisfies $\Psi(\xi) = \psi(\xi^2)$ 
  for a complete unbounded Bernstein function $\psi(\xi)$. If there exists $t_0>0$ such that \eqref{eq:qt:integcond} holds, then
  \begin{equation}
	\label{eq:qt:derivative}
    (-1)^n\dfrac{d^n}{dt^n} q_t(x) = \frac{2}{\pi}\int_0^\infty e^{-t\psi(\lambda^2)}F_\lambda(x)\lambda (\psi(\lambda^2))^n\sqrt{\psi'(\lambda^2)}d\lambda\/,\quad x>0\/,
  \end{equation}
	for every $t>t_0$ and $n=0,1,2,\ldots$. Moreover, if additionally
	\begin{enumerate}
	\item[(a)] $\psi$ is increasing, regularly varying of order $\alpha_0\in(0,1)$ at $0$, 
	then the following holds
	\begin{equation}
	\label{eq:qt:der:tlimit}
	\lim_{t\to \infty} \frac{t^{n+1}}{\sqrt{\psi^{-1}(1/t)}}\dfrac{d^n}{dt^n}q_t(x) = \frac{(-1)^n}{\pi}\Gamma\left(n+\frac{1}{2\alpha_0}-1\right)\, h(x)\/,\quad x\geq 0\/,
	\end{equation}
	where $\psi^{-1}$ denotes the inverse of $\psi$, and the convergence is locally uniform in $x$. This also holds for $\alpha_0=1$ with the additional assumption
	\begin{equation}
	 \label{eq:psi:cond}
	 \sup_{\xi>0}\frac{|\psi''(\xi)|}{\psi'(\xi)}<2\/.
	\end{equation}
	\item[(b)] $\psi$ is regularly varying at infinity with index $\alpha_\infty\in[0,1]$ and \eqref{eq:psi:cond} holds, then 
	\begin{equation}
	\label{eq:qt:der:xlimit}
	\lim_{x\to 0^+} \dfrac{d^n}{dt^n}q_t(x) = \frac{1}{\Gamma(1+\alpha_\infty)}\dfrac{d^n}{dt^n}\left(\frac{p_t(0)}{t}\right)\/,\quad t>t_0\/,
	\end{equation}
	where $p_t$ denotes the density of the transition semigroup of $(X,\mathbb{P})$.
	\end{enumerate}
\end{theorem}
\begin{proof}
To justify \eqref{eq:qt:derivative} it is enough to show that we can interchange the derivative and the integral in \eqref{eq:qt:formula}. However, taking any $t>t_0$, where $t_0$ is such that \eqref{eq:qt:integcond} holds, we can find $t_1\in(t_0,t)$ and a constant $c_1=c_1(t_0,t,n)$ such that
$$
  e^{-t\psi(\lambda^2)}(\psi(\lambda^2))^n\lambda\sqrt{\psi'(\lambda^2)}\leq c_1 e^{-t_1\psi(\lambda^2)} \lambda\sqrt{\psi'(\lambda^2)}
$$ 
and the claim follows from dominated convergence. 
Assuming additionally, that $\psi$ is increasing and regularly varying at $0$ with index $\alpha_0\in(0,1]$, we get that $\psi^{-1}$ is regularly varying (at $0$) with index $1/\alpha_0$ (see \cite{bib:Bingham}). Thus, there exists a constant $c_2=c_2(\alpha_0)>1$ such that 
\begin{equation}
\label{eq:qt:limit:bounds2}
\frac{1}{c_2}u^{1/(2\alpha_0)}\leq \frac{\psi(u/t)}{\psi(1/t)}\leq c_2 u^{2/\alpha_0}\/,\quad u>0
\end{equation}
and $t\in(0,1)$. Recall also \eqref{eq:Fl:limit}, which asserts that under the assumptions from point (a) the function $F_\lambda(x)/\sqrt{\lambda^2\psi'(\lambda^2)}$ extends to a continuous function for $\lambda\in[0,1]$. Here we use the upper-bound given in \eqref{eq:theta:upbound2} for $\alpha_0\in(0,1)$ and \eqref{eq:theta:upbound1} in the case $\alpha_0=1$ to show that $\lim_{\lambda\to 0^+}\vartheta_\lambda<\pi/2$, which is required to claim \eqref{eq:Fl:limit}. Consider the measure
$$
  \mu_t(d\lambda) = \frac{t^{n+1}}{\sqrt{\psi^{-1}(1/t)}} e^{-t\psi(\lambda^2)}\lambda^2(\psi(\lambda^2))^n\psi'(\lambda^2)\ind_{[0,1]}(\lambda)d\lambda
$$
and note that it tends to a point-mass at $0$ (as $t\to\infty$), as its density function tends to $0$ uniformly on $[\varepsilon,1]$, for every $\varepsilon>0$. The mass of $\mu_t$ can be calculated as follows
\begin{eqnarray*}
   ||\mu_t|| &=& \frac{t^{n+1}}{\sqrt{\psi^{-1}(1/t)}}\int_0^1e^{-t\psi(\lambda^2)}\lambda^2(\psi(\lambda^2))^n\psi'(\lambda^2)d\lambda\\
	&=& \frac{1}{2\sqrt{\psi^{-1}(1/t)}}\int_0^{t\psi(1)}e^{-u}u^{n}\sqrt{\psi^{-1}(u/t)}du\/.
\end{eqnarray*}
Since
$$
  \lim_{t\to\infty}\sqrt{\frac{\psi^{-1}(u/t)}{\psi^{-1}(1/t)}} = u^{1/(2\alpha_0)}\/,
$$
using \eqref{eq:qt:limit:bounds2} and dominated convergence we obtain
\begin{equation*}
   \lim_{t\to\infty}||\mu_t || = \frac{1}{2}\Gamma\left(n+\frac{1}{2\alpha_0}-1\right)\/.
\end{equation*}
Finally, the expression
\begin{equation*}
 \left| \int_1^\infty e^{-t\psi(\lambda^2)}F_\lambda(x)\lambda (\psi(\lambda^2))^n\sqrt{\psi'(\lambda^2)}d\lambda \right|
\end{equation*}
can by bounded for every $t>t_1$ by
\begin{equation*}
   2e^{(t-t_1)\psi(1)}\int_1^\infty e^{-t_1\psi(\lambda^2)}\lambda(\psi(\lambda^2))^n\sqrt{\psi'(\lambda^2)}d\lambda 
	\leq c_3e^{(t-t_1)\psi(1)}\int_1^\infty e^{-t_0\psi(\lambda^2)}\lambda\sqrt{\psi'(\lambda^2)}d\lambda\/,
\end{equation*}
with some $c_3=c_3(n,t_0,t_1)>0$, which together with the regularity of $\psi^{-1}$ at zero and estimates \eqref{eq:qt:limit:bounds2} implies that 
$$
  \frac{1}{\sqrt{\psi^{-1}(1/t)}}\int_1^\infty e^{-t\psi(\lambda^2)}F_\lambda(x)\lambda (\psi(\lambda^2))^n\sqrt{\psi'(\lambda^2)}d\lambda
$$
vanishes uniformly in $x$, as $t\to\infty$. Collecting all together we arrive at
\begin{equation*}
  \lim_{t\to \infty}\frac{(-1)^n}{\sqrt{\psi^{-1}(1/t)}}\dfrac{d^n}{dt^n}g_t(x) = \frac{1}{\pi}\Gamma\left(n+\frac{1}{2\alpha_0}-1\right)h(x)\/,\quad x\geq 0\/.
\end{equation*}

Because the justification of the fact that under assumption from point (b) we have
\begin{equation*}
   \lim_{x\to 0^+}\sqrt{\psi(1/x^2)}\dfrac{d^n}{dt^n}q_t(x) = \frac{(-1)^n}{\pi\Gamma(1+\alpha_\infty)}\int_0^\infty e^{-t\psi(\lambda^2)}(\psi(\lambda^2))^n\,\lambda^2 \psi'(\lambda^2)d\lambda\/,
\end{equation*}
follows in the same way as in the proof of Theorem 1.7 in \cite{bib:kmr13}, we omit the proof. Note that using \eqref{eq:qt:integcond} we can rewrite the last integral as
\begin{eqnarray*}
  (-1)^n \dfrac{d^n}{dt^n}\left(\int_0^\infty e^{-t\psi(\lambda^2)}\lambda^2 \psi'(\lambda^2)d\lambda\right) &=& (-1)^n \dfrac{d^n}{dt^n}\left(\frac{1}{t}\int_0^\infty e^{-t\psi(\lambda^2)}d\lambda\right)\/,
\end{eqnarray*}
where the last equality follows simply by integration by parts. Finally, the regular behavior of $\psi$ at infinity implies that $e^{-t\psi(\lambda^2)}$ is in $L_1(\R,d\lambda)$, which in particular means that the transition probability density is given by the inverse Fourier transform 
$$
p_t(x) = \frac{1}{2\pi} \int_0^\infty e^{-t\psi(\lambda^2)}e^{-ix\lambda}d\lambda\/.
$$
Combining all together we get \eqref{eq:qt:der:xlimit}, which ends the proof. 
\end{proof}

In addition to numerical applications of our results, they can be used to obtain more transparent representations as in the 
following example related to the Cauchy process.

\begin{proposition}
  \label{prop:Cauchy}
  For the symmetric Cauchy process, i.e. $\psi(\xi)= \sqrt{\xi}$, we have
	\begin{eqnarray*}
	   q_t(x) &=& \frac{1}{\sqrt{\pi}}\frac{\sin\left(\frac{\pi}{8}+\frac32 \arctan\left(\frac{x}{t}\right)\right)}{(t^2+x^2)^{3/4}}\\
		&& +\frac{1}{2\pi^{3/2}} 
	   \int_0^\infty \frac{y}{(1+y^2)(xy+t)^{3/2}}\exp\left(-\frac{1}{\pi}\int_0^\infty \frac{\log(y+s)}{1+s^2}\right)dy\/,
	\end{eqnarray*}
	where $x,t>0$.
\noindent Then the density $f_t(x)$ of the past supremum at time $t$ of $(X,\mathbb{P})$  can be derived from the above 
expression together with $(\ref{eq:ft:formula})$ and $(\ref{eq:sym:n})$.
	\end{proposition} 
	\begin{proof}
	   Since $\psi(\xi)=\xi^{1/2}$, $\psi'(\xi)=1/(2\sqrt{\xi})$ the formula \eqref{eq:qt:formula} reads as
		\begin{equation*}
		  q_t(x) = \frac{\sqrt{2}}{\pi}\int_0^\infty e^{-t\lambda} F_x(\lambda)\sqrt{\lambda}d\lambda\/,
		\end{equation*}
		where we used the scaling property $F_\lambda(x)=F_1(\lambda x)= F_x(\lambda)$.	By the Plancherel's theorem we get, for fixed $b\in (0,t)$, that 
		\begin{eqnarray}
		\label{eq:qt:Cauchy}
		   q_t(x) &=& \int_0^\infty \left(e^{-b\lambda}F_x(\lambda)\right)\left(e^{-(t-b)\lambda}\frac{\sqrt{2\lambda}}{\pi}\right)d\lambda \nonumber\\
			&=& \frac{1}{2\pi} \int_{-\infty}^\infty \mathcal{L}F_x(b+is) \overline{\mathcal{L}f(t-b+is)}ds\nonumber\\
			&=& \frac{1}{2\pi i} \mathcal{L}F_x(b+is)\int_{b-i\infty}^{b+i\infty} \mathcal{L}F_x(z)\mathcal{L}f(t-z)dz\/,
		\end{eqnarray}
		where $f(x)=\sqrt{2x}/\pi$. The Laplace transform of $f$ can easily be computed as follows 
		$$
		  \mathcal{L}f(z) = \frac{\sqrt{2}}{\pi}\int_0^\infty e^{-z x}\sqrt{x}dx = \frac{1}{\sqrt{2\pi}}\frac{1}{z^{3/2}}\/,\quad \re(z)>0\/.
		$$
    Formula \eqref{eq:Fl:laplacetransform} gives
		$$
		\mathcal{L}F_x(z) = \frac{1}{\sqrt{2}}\frac{x}{x^2+z^2}\exp\left(\frac{1}{\pi}\int_0^\infty \frac{z\log(1+u/x)}{z^2+u^2}du\right)\/.
		$$
		Substituting $u=z/s$ in the last integral we get
		\begin{eqnarray*}
		    \int_0^\infty \frac{z}{z^2+u^2}\log(1+u/x) du &=& \int_0^\infty \frac{\log\left(1+\frac{z}{xs}\right)}{1+s^2}\,ds
				= \int_0^\infty \frac{\log(z/x+s)}{1+s^2}\,ds\/,
		\end{eqnarray*}
		where the last equality follows from the fact that 
		\begin{eqnarray*}
			\int_0^\infty \frac{\log s\,ds}{1+s^2}=\left(\int_0^1+\int_1^\infty\right)\frac{\log s\,ds}{1+s^2}=\int_0^1 \frac{\log s\,ds}{1+s^2}+\int_0^1 \frac{\log(1/s)\,ds}{1+s^2} = 0\/.
		\end{eqnarray*}
		Finally, the function 
	  $$
		  B(z) = \frac{1}{\pi} \int_0^\infty \frac{\log(z+u)}{1+u^2}du\/,\quad z\in \C\setminus (-\infty,0]\/,
		$$
		studied in details in \cite{bib:kkms:2010}, is holomorphic in the region. We recall (see (3.13) in \cite{bib:kkms:2010}) that 
		\begin{eqnarray}
		\label{eq:Bi}
		   B(i) = \frac{\log 2}{2}+i\frac{\pi}{8}
		\end{eqnarray}
		and (see (4.1) in \cite{bib:kkms:2010}) that
		\begin{eqnarray}
		\label{eq:B:jump}
		  e^{B(z)} = (1-iz\sigma(z))e^{-B(-z)}\/,
		\end{eqnarray}
		where $\sigma(z)=1$ for $\im(z)>0$ and $\sigma(z)=-1$ for $\im(z)<0$. Consequently, defining (for fixed $x$) the function of complex variable $z$ 
		$$
		  G_x(z) = \frac{x}{x^2+z^2}\frac{1}{(t-z)^{3/2}} \exp\left(\frac{1}{\pi}\int_0^\infty \frac{\log(z/x+u)}{1+u^2}du\right) = \frac{x}{x^2+z^2}\frac{1}{(t-z)^{3/2}}e^{B(z/x)}\/,
		$$
		it is easy to see that $G_x(z)$ is  a meromorphic function on $\{z\in \C: \re(z)<t\}\setminus(-\infty,0]$ with single poles at $ix$ and $-ix$. To evaluate the integral \eqref{eq:qt:Cauchy} we integrate $G_x$ over the (positively oriented) curve consisting of (see figure \ref{fig:contour} below)
		\begin{itemize}
		   \item[(i)] four horizontal segments: 
			\begin{enumerate}
			   \item[] $\gamma_1 = \{z: \im(z)=n,\re(z)\in[-n,b]\}$, 
			   \item[] $\gamma_2 = \{z: \im(z)=-n,\re(z)\in[-n,b]\}$,
				 \item[] $\gamma_3=\{z: \im(z)=1/n,\re(z)\in[-n,0]\}$,
				 \item[] $\gamma_4=\{z: \im(z)=-1/n,\re(z)\in[-n,0]\}$;
			\end{enumerate}
			\item[(ii)] three vertical segments: 
			 \begin{enumerate}
			   \item[] $\gamma_5=\{z: \re(z)=-n,\im(z)\in[1/n,n]\}$,
				 \item[] $\gamma_6=\{z: \re(z)=-n,\im(z)\in[-n,-1/n]\}$, 
				 \item[] $\Gamma = \{z: \im(z)=b,\re(z)\in[-n,n]\}$
				\end{enumerate}
			\item[(iii)] a semi-circle: $\gamma_7=\{z: |z|=1/n, \Re(z)\geq 0\}$.
		\end{itemize}
		\begin{figure}[h]
		\label{fig:contour}
    \begin{picture}(256,256)(0,0)
      \put(0,128){\vector(1,0){256}}
      \put(128,0){\vector(0,1){256}}
      \linethickness{1pt}
      \put(128,128){\oval(5,5)[r]}
      \put(128,126){\line(-1,0){100}}
      \put(128,130){\line(-1,0){100}}
      \put(28,126){\line(0,-1){80}}
      \put(28,130){\line(0,1){80}}
      \put(28,210){\line(1,0){170}}
      \put(28,46){\line(1,0){170}}
      \put(198,46){\line(0,1){164}}
      \put(201,123){$_b$}\put(130,40){$_{-in}$}
      \put(130,213){$_{in}$}\put(15,123){$_{-n}$}
      \linethickness{0.5pt}
      \put(195,138){$\vector(0,1){10}$}
			\put(100,215){$_{\gamma_1}$} \put(100,42){$_{\gamma_2}$}
			\put(18,165){$_{\gamma_5}$} \put(18,92){$_{\gamma_6}$}
      \put(100,135){$_{\gamma_3}$} \put(100,120){$_{\gamma_4}$}
      \put(130,134){$_{\gamma_7}$} 
      \put(202,170){$_{\Gamma}$}
      \put(80,135){$\vector(1,0){10}$} \put(90,120){$\vector(-1,0){10}$}
      \put(128,168){\circle*{2}}\put(128,88){\circle*{2}}
      \put(118,168){$_{ix}$}\put(112,88){$_{-ix}$}
			\put(230,128){\circle*{2}}\put(228,123){$_{t}$}
     \end{picture}
		\caption{The contour of integration.}
     \end{figure}
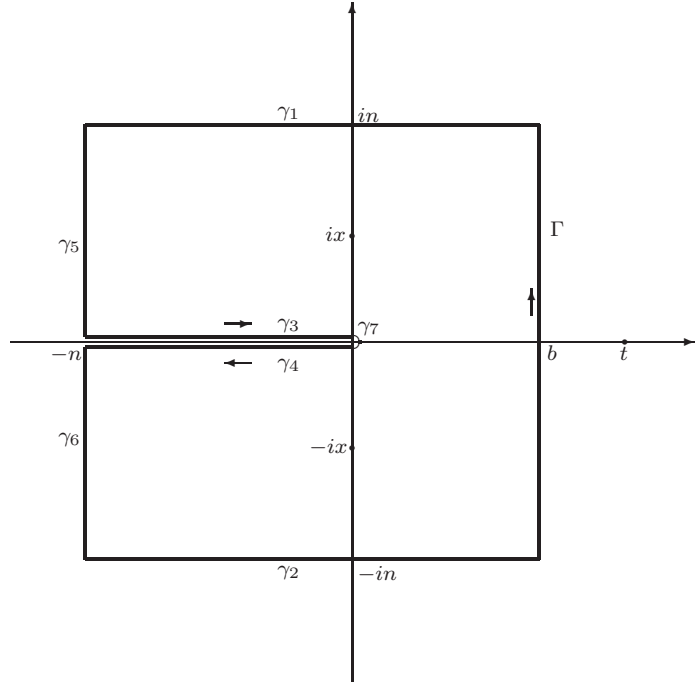
	First we compute the residua of $G_x$ at $ix$ and $-ix$. By \eqref{eq:Bi}, we have
	\begin{equation*}
	    \textrm{Res}\left(G_x,ix\right) = \frac{1}{2i} \frac{\sqrt{2}e^{i\pi/8}}{(t-ix)^{3/2}} \/,\quad\quad
			\textrm{Res}\left(G_x,-ix\right) = -\frac{1}{2i} \frac{\sqrt{2}e^{-i\pi/8}}{(t+ix)^{3/2}} \/.
	\end{equation*}
	Since 
	$
	(t\pm ix)^{3/2} = (t^2+x^2)^{3/4}e^{\pm 3i/2\arctan(x/t)}\/,
	$
	we arrive at
	\begin{eqnarray}
	  \label{eq:res}
	\textrm{Res}\left(G_x,ix\right)+\textrm{Res}\left(G_x,-ix\right) &=& \frac{\sqrt{2}\sin \left(\pi/8+\frac{3}{2}\arctan(x/t)\right)}{(t^2+x^2)^{3/4}}\/. 
	\end{eqnarray}
	Using the relation \eqref{eq:B:jump} we obtain
	\begin{eqnarray*}
	    \left(\int_{\gamma_3}+\int_{\gamma_4}\right)G_x(z)dz &\stackrel{n\to \infty}{\longrightarrow}& \int_{-\infty}^0 \frac{x((1-iu/x)-(1+iu/x))e^{-B(-u/x)}}{(x^2+u^2)(t-u)^{3/2}}\,du\/,
	\end{eqnarray*}
	where the last integral, after substituting $y=-xu$, is equal to 
	\begin{eqnarray}
	   \label{eq:inegral1}
	   -2i \int_0^\infty \frac{y}{1+y^2}\frac{1}{(t+yx)^{3/2}}\exp\left(-\frac{1}{\pi}\int_0^\infty \frac{\log(y+u)}{1+u^2}du\right)\,dy\/.
	\end{eqnarray}
   Using the bounds \eqref{eq:Fl:laplaceestimate} we can write
	\begin{eqnarray*}
	   |G_x(z)| = 2\sqrt{\pi}|\mathcal{L}F_x(z)\mathcal{L}f(t-z)|\leq c_1 \frac{|x+z|}{|x^2+z^2|}\frac{1}{|t-z|^{3/2}}\leq c_2\,(\im z)^{-5/2}\/.
	\end{eqnarray*}
 It implies that the integrals of $G_x(z)$ over $\gamma_1$, $\gamma_2$, $\gamma_5$ and $\gamma_6$ vanish as $n$ goes to infinity. Since $G_x(z)$ is bounded in the neighborhood of $0$ ($\re(z)>0$), the same holds for the integral over the semi-circle $\gamma_7$. Now we can finish the computations by applying the residue theorem in order to get
\begin{eqnarray*}
  \frac{1}{2\pi i} \lim_{n\to \infty}\int_{-n}^n G_x(z)dz = \textrm{Res}\left(h_x,ix\right)+\textrm{Res}\left(G_x,-ix\right)-\lim_{n\to \infty}  \left(\int_{\gamma_3}+\int_{\gamma_4}\right)G_x(z)dz\/.
\end{eqnarray*}
Taking into account \eqref{eq:inegral1} and \eqref{eq:res} and dividing both sides by $2\sqrt{\pi}$ lead to the result. 
	\end{proof}
\begin{remark}
  It is also possible to find similar formula for the entrance law density of the symmetric $\alpha$-stable process with index $\alpha\in(0,1)$. 
  Using the scaling property $F_{\lambda}(x)=F_1(\lambda x)$ and writing 
	$$
	   e^{-t\lambda^{\alpha}} = \int_0^\infty e^{-u\lambda} g_t^{(\alpha)}(u)du\/,\quad t>0\/,
	$$
	where $g_t^{(\alpha)}(u)$ is the density of the $\alpha$-stable subordinator we obtain
	$$
	  q_t(x) = \frac{\sqrt{2\alpha}}{\pi} \int_0^\infty \left(\int_0^\infty e^{-u\lambda}F_x(\lambda)\,\lambda^{\alpha/2}d\lambda\right)g_t^{(\alpha)}(u)du\/.
	$$
	The inner integral can be evaluated similarly as in Proposition \ref{prop:Cauchy}. 
\end{remark}

\section{Stable processes}
\label{sec:stable}

For the rest of the paper we focus on stable processes and use the theory of the corresponding generalized eigenfunctions developed in \cite{bib:kk18}. We assume that $X$ is a stable process with characteristic exponent
$$
  \Psi(\xi) = |\xi|^\alpha e^{\pi i \alpha (1/2-\rho)\textrm{sign}(\xi)}\/,\quad \xi\in\R\/.
$$
We exclude spectrally one-sided processes from our considerations, i.e.~we assume that $\alpha\in(0,1]$ and $\rho\in(0,1)$ or $\alpha\in(1,2]$,
but then we assume that $\rho\in(1-1/\alpha,1/\alpha)$. We write $\rho^* = 1-\rho$ and define non-symmetric analogous of $F_1(x)$ defined in Section \ref{sec:sym} for stable processes as follows
\begin{eqnarray}
   \label{eq:stable:F}
   F(x) = e^{\pi \cos(\pi \rho)}\sin(x\,\sin(\pi \rho)+\pi \rho(1-\rho^*)/2)+\frac{\sqrt{\alpha}}{4\pi}S_2(-\alpha \rho^*)G(x)\/,
   \end{eqnarray}
where
\begin{eqnarray}
   \label{eq:stable:G}
   G(x) = \int_0^\infty e^{-zx}z^{\alpha\rho/2-1/2}|S_2(1+\alpha+\alpha\rho^*/2+i\alpha\ln(z)/(2\pi))|^2dz\/.
\end{eqnarray}
The function $S_2(z)=S_2(z;\alpha)$ is the double sine function uniquely determined by the following functional equations
$$
  S_2(z+1)=\frac{S_2(z)}{2\sin(\pi z/\alpha)}\/,\quad S_2(z+\alpha)=\frac{S_2(z)}{2\sin(\pi z)}
$$
together with the normalizing condition $S_2((1+\alpha)/2)=1$ 
(see
\cite{bib:KoKu2006}, \cite{bib:KoKu2007} and Appendix A in \cite{bib:kk18} for equivalent definitions and further properties). We define $F^*(x)$ and $G^*(x)$ by the same formulae as in \eqref{eq:stable:F} and \eqref{eq:stable:G} but with $\rho$ replaced by $\rho^*$ (and consequently $\rho^*$ replaced by $(\rho^*)^*=\rho$). Note that whenever $\rho> 1/2$ the oscillations of $F$ coming from the sine function are multiplied by the exponentially decreasing factor, but then $F^*$ oscillates exponentially, when $x\to \infty$ and the situation is reversed for $\rho<1/2$. The behaviour of $F$ at zero is described by 
(see the proof of Lemma 2.8 in \cite{bib:kk18})
\begin{equation}
 \label{eq:stable:F:zero}
   F(x) = \frac{\sqrt{\alpha}}{2}\frac{S_2(\alpha\rho)}{\Gamma(1+\alpha\rho^*)}\cdot x^{\alpha\rho^*}(1+o(1))\/,\quad x\to 0^+\/.
\end{equation} 
Although the constant $\frac{\sqrt{\alpha}}{2}\frac{S_2(\alpha\rho)}{\Gamma(1+\alpha\rho^*)}$ was not specified in \cite{bib:kk18}, using (1.10) and (1.19) from \cite{bib:kk18}, we obtain that 
$$
  \int_0^\infty e^{-zx}F(x)dx = \frac{\sqrt{\alpha}}{2}S_2(\alpha\rho)z^{-1-\alpha\rho^*}(1+o(1))\/,\quad z\to \infty\/.
$$
Consequently, using the Karamata's Tauberian theorem and the Monotone Density Theorem we obtain \eqref{eq:stable:F:zero}. Moreover, if $\rho>1/2$ then 
\begin{equation}
   \label{eq:stable:F:infty}
	 F(x) = G(x) = O(x^{-\alpha-1})\/,\quad F^*(x) = O(e^{x\cos(\pi\rho^*)})\/,\qquad x\to \infty\/.
\end{equation}
Even though the functions $F$ and $F^*$ do not simultaneously belong to $L^2(0,\infty)$ (for $\rho\neq 1/2$), they can be understood as the generalized eigenfunctions of the semigroups $\mathbf{Q}_t^*$ and $\mathbf{Q}_t$, respectively (see Theorem 1.3 in \cite{bib:kk18}). Moreover, using Theorem 1.1 in \cite{bib:kk18} the transition probability density of the process $X$ killed when exiting the positive half-line is given by
\begin{equation}
   \label{eq:stable:qtxy}
	  q_t^*(x,y) = \frac{2}{\pi} \int_0^\infty e^{-t\lambda^\alpha} F(\lambda x)F^*(\lambda y)d\lambda\/,\quad x,y,t>0\/,
\end{equation}
whenever $\alpha>1$. Note that the restriction on $\alpha$ ensures that the exponential oscillations of $F$ (or $F^*$) are suppressed by the factor $e^{-t\lambda^{\alpha}}$, which makes the integral convergent. The formula \eqref{eq:stable:qtxy} is the analogue of the integral representation for subordinate Brownian motions presented in \cite{bib:mk10}. Note also that assuming $\rho=1/2$ we have $F(x)=F^*(x)=F_1(x)$, where $F_\lambda(x)$ is the generalized eigenfunction defined in Section \ref{sec:sym} for symmetric $\alpha$-stable process.

In the next theorem we present a relation between the entrance laws densities $q_t^*(x)$, $q_t(x)$ and the functions $F$, $F^*$.

\begin{theorem}
    Let $(X,\pr)$ be a stable process with parameter $\alpha>1$ and $\rho\in (1-1/\alpha,1/\alpha)$ or $\rho=1/2$. Then 
		\begin{equation*}
		   q_t(x) = \frac{\sqrt{\alpha}}{\pi}\,S_2(\alpha\rho^*)\int_0^\infty e^{-t\lambda^{\alpha}}F(\lambda x)\lambda^{\alpha \rho}d\lambda\/,\,\,
			 q_t^*(x) = \frac{\sqrt{\alpha}}{\pi}\,S_2(\alpha\rho)\int_0^\infty e^{-t\lambda^{\alpha}}F^*(\lambda x)\lambda^{\alpha \rho^*}d\lambda\/,
		\end{equation*}
		for every $x,t>0$.
		
\end{theorem} 
\begin{proof}
  We will exploit formula \eqref{eq:stable:qtxy} together with the relation (see Proposition 1 in \cite{bib:chm16})
	\begin{equation}
	   \label{eq:qtxy:limit}
	   \lim_{x\to 0^+}\frac{q_t^*(x,y)}{h^*(x)} = q_t^*(y)\/, \quad y,t>0\/, 
	\end{equation}
	where the renewal function $h^*(x)$  of the ladder height process $H^*$ is defined, in general, by $h^*(x) = \int_0^\infty \pr(H_t\leq x)dt$. In the stable case $H_t^*$ is the $\alpha\rho^*$ stable subordinator and  
	$$
	  h^*(x) = \ex H_1^{-\alpha\rho^*}\cdot x^{\alpha\rho^*} = \frac{x^{\alpha\rho^*}}{\Gamma(1+\alpha\rho^*)}\/,\quad x\geq 0\/.
	$$
	Choosing $c>0$ small enough and using \eqref{eq:stable:F:zero} we can write 
	\begin{eqnarray*}
	   \ind_{\{\lambda x<c\}}e^{-t\lambda^\alpha}\frac{F(\lambda x)}{x^{\alpha\rho^*}}F^*(\lambda y)\leq c_1 e^{-t\lambda^\alpha} \lambda^{\alpha\rho^*}F^*(\lambda y)\/,
	\end{eqnarray*}
	where the latter function is integrable over $(0,\infty)$ (for fixed $t$ and $y$) by \eqref{eq:stable:F:infty}. Thus, by the Lebesgue dominated convergence theorem and \eqref{eq:stable:F:zero} we arrive at
	\begin{equation*}
	    \lim_{x\to 0^+}\frac{1}{x^{\alpha \rho^*}}\int_0^{c/x}e^{-t\lambda^\alpha}\frac{F(\lambda x)}{x^{\alpha\rho^*}}F^*(\lambda y)d\lambda = \frac{\sqrt{\alpha}}{2}\frac{S_2(\alpha\rho)}{\Gamma(1+\alpha\rho^*)} \int_0^\infty e^{-t\lambda^\alpha}F^*(\lambda y)\lambda^{\alpha \rho^*}d\lambda\/.
	\end{equation*}
	Moreover, by \eqref{eq:stable:F:infty}, we can write for $x<1$ that
	\begin{eqnarray*}
	   \int_{c_1/x}^\infty e^{-t\lambda^\alpha} |F(\lambda x)F^*(\lambda y)|d\lambda &\leq& \int_{c_1/x}^\infty e^{-t\lambda^\alpha}e^{\lambda(x\vee y)\cos(\pi(\rho\vee\rho^*))}d\lambda\\
		&\leq& \exp\left(-\frac{tc_1^{\alpha}}{2x^{\alpha}}\right)\int_{c_1}^\infty  e^{-t\lambda^\alpha/2} e^{\lambda(1\vee y)\cos(\pi(\rho\vee\rho^*))}d\lambda\/,
	\end{eqnarray*}
	where the last integral is convergent according to our assumption $\alpha>1$\/.
	It shows that 
	\begin{equation*}
	   \lim_{x\to 0^+}\frac{1}{x^{\alpha\rho^*}}\int_{c_1/x}^\infty e^{-t\lambda^\alpha}F(\lambda x)F^*(\lambda y)d\lambda = 0
	\end{equation*}
	and consequently, by \eqref{eq:qtxy:limit}, we obtain
	\begin{eqnarray*}
	   q_t^*(y) &=& \frac{2\Gamma(1+\alpha\rho^*)}{\pi}\lim_{x\to 0^+}\left(\frac{1}{x^{\alpha\rho^*}}\int_0^\infty e^{-t\lambda^{\alpha}}F(\lambda x)F^*(\lambda y)d\lambda\right)\/,\\
		&=& \frac{\sqrt{\alpha}}{\pi}\,{S_2(\alpha\rho)}\int_0^\infty e^{-t\lambda^\alpha}F^*(\lambda y)\lambda^{\alpha\rho^*}\,d\lambda\/, \quad y,t>0\/.
	\end{eqnarray*}
By duality, we have the corresponding integral representation for $q_t(x)$ with $F^*(x)$ and $\rho^*$ replaced by $F(x)$ and $\rho$.
\end{proof}
The analogue of Theorem \ref{thm:sym:distributions} can now be proved. 

\begin{theorem}
\label{thm:stable:distributions}  Let $(X,\pr)$ be a stable process with parameter $\alpha>1$ and $\rho\in (1-1/\alpha,1/\alpha)$ or $\rho=1/2$. The density of  $(\overline{X}_t,\overline{X}_t-X_t)$  with respect to the Lebesgue measure $dxdy$ on $(0,\infty)^2$ is given by
\begin{equation*}
   \frac{2\alpha\sin(\pi\rho^*)}{\pi^2}\iint_{(0,\infty)^2}\frac{e^{-t\lambda^\alpha}-e^{-tu^\alpha}}{\lambda^\alpha-u^\alpha}F(u y)F^*(\lambda x)\lambda^{\alpha \rho}u^{\alpha\rho^*}du\,d\lambda\/.
\end{equation*}
Moreover, we have
\begin{equation*}
   f_t(x) = \frac{\sqrt{\alpha}}{\pi}\frac{S_2(\alpha\rho)}{\Gamma(\rho)}\int_0^\infty e^{-t\lambda^\alpha}\left(\int_0^{t\lambda^\alpha}\frac{e^udu}{u^{\rho^*}}\right)F^*(\lambda x)\,d\lambda\/,
\end{equation*}
for every $t,x>0$.
\end{theorem}
\begin{proof}
As previously, the result follows from  the integral representations for $q_t(x)$ and $q_t^*(x)$, the relations \eqref{eq:ftpt:formula}, \eqref{eq:ft:formula} and Fubini's theorem. However, since the ladder time process $(L_t^*)^{-1}$ is $\rho^*$-stable subordinator and 
$n^*(t<\zeta)=\pi^*(t,\infty)$, where $\pi^*$ is the L\'evy measure of $(L_t^*)^{-1}$ we have
$$
  n(t<\zeta) = \frac{1}{\Gamma(1-\rho^*)t^{\rho^*}}\/,\quad t>0\/,
$$
which gives the representations for $f_t(x)$. To find the constant in the other formula we use the relations $S_2(z)S_2(1+\alpha-z)=1$ and $2\sin(\pi z/2)S_2(z+1)=S_2(z)$ (see (A.7) and (1.9) in \cite{bib:kk18}).
\end{proof}

\noindent Recall that spectrally one sided L\'evy processes are excluded in Theorem \ref{thm:stable:distributions}. However let us note that in this case, some expressions of the law of $(\overline{X}_t,\overline{X}_t-X_t)$ 
in terms of the density of $X_t$ are given in Theorem 3.1 in \cite{bib:chm18}.

\end{document}